\documentclass[11pt]{amsart}
\usepackage{graphicx}
\usepackage{amsmath,amsthm,amsfonts,amsopn,verbatim}
\usepackage{amssymb}
\usepackage{mathrsfs}
\usepackage[tight]{subfigure}
\usepackage{epsfig,epstopdf,psfrag}
\usepackage{hyperref,bm,color,fancyhdr}
\usepackage{algorithm,algorithmic}
\usepackage{array}
\usepackage{booktabs}
\usepackage{natbib}

\definecolor{purple}{rgb}{0.65, 0, 1}
\definecolor{green}{rgb}{0, 1, 0}
\definecolor{orange}{rgb}{1,.5,0}

\newenvironment{green}{\par\color{green}}{\par}
\newenvironment{orange}{\par\color{orange}}{\par}

\newcommand{\be}{\begin{equation}}
\newcommand{\ee}{\end{equation}}
\newcommand{\bb}{\bigskip}
\newcommand{\sm}{\smallskip}
\newcommand{\la}{\label}
\newcommand{\rf}[1]{(\ref{#1})}

\newcommand{\om}{\omega}
\newcommand{\Om}{\Omega}

\newcommand{\R}{{\mathbf R}}

\newcommand{\pul}{\frac{1}{2}}

\newtheorem{theorem}{Theorem}
\newtheorem{lemma}[theorem]{Lemma}
\newtheorem{prop}[theorem]{Proposition}

\newenvironment{remark}{\noindent \textbf{Remark}.}{\hfill $\square$}

\newcommand{\curl}{\mathop{\rm curl}\nolimits}

\newcommand{\PV}{\mathop{\rm P.V.}\nolimits}
\newcommand{\re}{\mathrm{e}}

\providecommand{\abs}[1]{\lvert#1\rvert}
\providecommand{\norm}[1]{\lVert#1\rVert}

\numberwithin{equation}{section}
\numberwithin{theorem}{section}
\numberwithin{figure}{subsection}
\numberwithin{table}{subsection}

\title[Blowup of a 1D Model for 3D Euler]{On the Finite-Time Blowup of a 1D Model for the 3D Axisymmetric Euler Equations}

\author{K. Choi}
\thanks{Department of
Mathematics, University of Wisconsin, Madison; email:
kchoi@math.wisc.edu}
\author{T. Hou}
\thanks{California Institute of
Technology; email: hou@cms.caltech.edu}
\author{A. Kiselev}
\thanks{Rice University; email: kiselev@rice.edu
}
\author{G. Luo}
\thanks{California Institute of
Technology; email: gluo@caltech.edu;
present affiliation: City University of Hong Kong
 }
\author{V. Sverak}
\thanks{University of Minnesota; email: sverak@math.umn.edu}
\author{Y. Yao}
\thanks{University of Wisconsin, Madison; email:
yaoyao@math.wisc.edu}


\begin{document}
\maketitle


\section*{Abstract}
In connection with the recent proposal for possible singularity formation at the boundary for solutions of 3d axi-symmetric incompressible Euler's equations~\citep{lh2014a}, we study models for the dynamics at the boundary and show that they exhibit a finite-time blow-up from smooth data.

\section{Introduction}\label{sec_intro}

In this paper we study a 1d model introduced in~\cite{lh2014a}(also see \cite{lh2014b}) in connection with incompressible Euler equations. The study of 1d models for hydrodynamical equations has a long history, going back to the works of~\cite{burgers} and~\cite{hopf}. A slightly different type of 1d models, originating, as far as we know, with the paper of \cite{clm1985}, were introduced to illustrate effects of vortex stretching. The equations studied in this paper are related to these works and we will link some of the models to the boundary behavior of fluid flows. The link is not perfect, but is still useful.

Important modifications of the original model of \cite{clm1985} were proposed by \cite{degreg1990, degreg1996} and somewhat differently motivated work of~\cite{ccf2005}. All these works include modeling of two important features of incompressible flow:\\
(i) the vorticity transport (either as a scalar, as in 2d Euler, or as a vector field, as in 3d Euler, with the vector field transport also covering the vortex stretching) and\\
(ii) the Biot-Savart law, which expresses the velocity field which transports the vorticity in terms of the vorticity itself.\\
 In the present paper the ``vorticity" $\om$ will be mainly considered as a scalar function $\om(x,t)$ of time and a 1d variable $x$, either on the real line $\R$, or on the 1d circle ${\bf S}^1$, the latter case corresponding to periodic boundary conditions.
The Biot-Savart law will be taken as
\be\la{bslaw1}
u_x=H\om\,,
\ee
where $H$ is the Hilbert transform. In this setting one can consider the equation
\be\la{2dem}
\om_t+u\om_x=0\,,
\ee
which has many properties similar to the 2d Euler.
We will see below that a natural (approximate) interpretation of this model is in terms of the dynamics at the boundary for the full 2d Euler flows in smooth domains with boundaries.

Model~\rf{2dem},~\rf{bslaw1} is not studied in the works mentioned above. However, one can prove that it shares many properties with the 2d Euler equation, including the global existence and uniqueness of solutions (in suitable classes) for $L^\infty$ initial data, as in \cite{yudovich1963}, and the double exponential growth of $\om_x$ for certain smooth data, similar to~\cite{kv2014}. These topics will be addressed elsewhere. Here we will focus on singularity formation for natural extensions of the ``2d~Euler model", which in some sense take us from 2d Euler to 3d axi-symmetric Euler with swirl or 2d Boussinesq.

 In~\cite{ccf2005} the Biot-Savart law is taken as
\be\la{ccfbs}
u=H\om\,.
\ee
With this Biot-Savart law, equation~\rf{2dem} is more akin to the surface quasi-geostrophic (SQG) which was proposed as a model for 3d Euler in~\cite{cmt1994}, under a slightly different terminology. While the singularity formation for the SQG equation remains open, it was shown in~\cite{ccf2005} that the 1d model can develop a singularity from smooth initial data in finite time.

In~\cite{degreg1990, degreg1996}, the Biot-Savart law is taken as in~\rf{bslaw1}, but the vorticity is considered as a vector field $\om(x,t) \frac{\partial}{\partial x}$, and transported by the velocity field $u(x,t)$ as such (similarly to what we have for 3d Euler), with the transport equation given by
\be\la{dgeq}
\om_t+u\om_x=u_x\om\,.
\ee
We can also write it in terms of the usual Lie bracket for vector fields as
\be\la{dglie}
\om_t+[u,\om]=0\,,
\ee
 just as the 3d Euler. It appears that the question of global existence of smooth solutions for smooth initial data for the 1d model~\rf{dglie} with~\rf{bslaw1} is open.
A generalization of the model \eqref{dgeq} was studied in \citet{osw2008} and also in \citet{cc2010}, see the table below.

As pointed out in~\cite{degreg1990}, the model considered in~\cite{clm1985} can be written as
\be\la{clmeq}
\om_t=u_x\om\,,
\ee
 with the Biot-Savart law~\rf{bslaw1}, which is of course the same as~\rf{dgeq} without the ``transport term" (in the scalar sense) $u\om_x$. As shown by~\cite{clm1985}, this model can blow up in finite time from smooth data.

In this paper we will mostly study the model obtained from~\rf{bslaw1} and~\rf{2dem} by adding an additional variable $\theta=\theta(x,t)$ (which can be thought of as temperature in the 2d Boussinesq context or the square of the swirl component $u^\theta$ of the velocity field in the 3d axi-symmetric case)
and considering equations
\begin{subequations}\label{hl}
\begin{align}
\om_t+u\om_x & = \theta_x \,\,,\\
\theta_t + u\theta_x & = 0\,.
\end{align}
\end{subequations}

 A discussion of connections to 3d axi-symmetric Euler flows with swirl, or 2d Boussinesq flows is included in the next section. Aspects of the model have been discussed in~\cite{hl2013} and we will refer to the model as the HL-model. One of our main results is that this model can exhibit finite-time blow-up from smooth initial data.

In addition to this model, we will also discuss its variant where the Biot-Savart law~\rf{bslaw1} is replaced by a simplified version which still captures important features of~\rf{bslaw1} in the situation when $\om(x,t)$ is odd in $x$:
\be\la{bscky}
u(x)\sim-x\int_x^\infty \frac{\om(y)}{y}\,dy\,.
\ee
This model was studied by~\citep{cky2015}, and we will refer to it as the CKY-model.

Our proof of finite time blow up for the CKY-model relies on the local aspect still present in the
simplified Biot-Savart law - namely, that $(u(x)/x)_x = \omega(x)/x$. This observation allows us to
construct an entropy-like functional to capture the blow up. The proof of finite time blow up
for the full HL-model is subtler, as the Biot-Savart law is now fully nonlocal. The proof is based, at heart,
on a non-trivial and not readily evident sign definite estimate for certain quadratic forms associated
with the solution. These estimates are novel and intriguing to us, as they provide a blueprint of what one might
have to look for in order to understand the full three dimensional dynamics.

\sm
We summarize the above discussion in a table.

\sm
{\small
\begin{center}
    \begin{tabular}{ | p{5cm} | l | l | p{4cm} |}
    \hline
    model & Biot-Savart law & dynamical equations & regularity of the solutions of the model \\ \hline
    2d Euler analogy & $u_x=H\om$ & $\,\,\om_t+u\om_x=0$ & unique global solutions
     \\
     \hline
    \hbox{\cite{degreg1990},} \hbox{3d Euler analogy}& $u_x=H\om$ & $\,\,\om_t+u\om_x=u_x\om$ & global existence and regularity unknown\\
    \hline
    \hbox{\cite{clm1985},} \hbox{analogy of vortex stretching} without transport term& $u_x=H\om$ & $\,\,\om_t=u_x\om$ & \hbox{finite-time blow-up} from smooth data \\
     \hline
      { \hbox{\citet{osw2008},} \hbox{a generalized model }} &
    $u_x=H\om$ &
    $
    \begin{array}{rcl}
    \om_t+au\om_x & =& u_x\om
        \end{array}
    $

     & {\footnotesize finite-time~blowup~when \hbox{$a< 0$}, \cite{cc2010}}\\
    \hline
    \hbox{\cite{ccf2005},} \hbox{SQG analogy} &
    $u=H\om$ & $\,\,\om_t+u\om_x=0$ & \hbox{finite-time blow-up} from smooth data\\
    \hline
    HL-model, \hbox{\cite{hl2013},}
    \hbox{2d Boussinesq}/\hbox{3d axi-symmetric} Euler analogy &
    $u_x=H\om$ &
    $
    \begin{array}{rcl}
    \om_t+u\om_x & =& \theta_x\\
    \theta_t+u\theta_x & = & 0
    \end{array}
    $
& finite time blow-up from smooth data, the main new result of this paper\\
     \hline
     CKY-model, \hbox{\cite{cky2015},} simplified HL-model &
     $u(x)=-x\int_x^\infty \frac{\om(y)}{y}\,dy
     $
     &
    $
    \begin{array}{rcl}
    \om_t+u\om_x & =& \theta_x\\
    \theta_t+u\theta_x & = & 0
    \end{array}
    $
    &
    \hbox{finite-time blow-up} from smooth data\\
    \hline
\end{tabular}
\end{center}
}


\section{Motivation for the models}\label{motivation}

In~\cite{lh2014a} the authors study 3d axi-symmetric flow of incompressible Euler's equation with roughly the following initial configuration:

\begin{center}
\includegraphics[width=100mm]{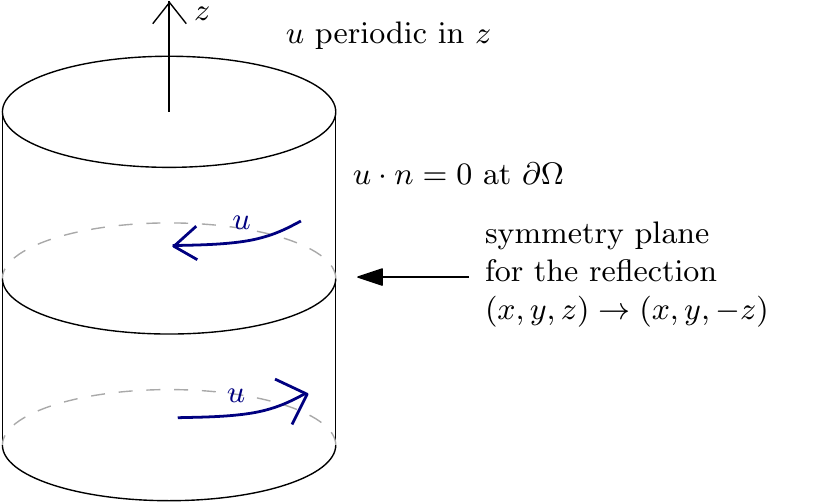}
\end{center}

In accordance with principles governing the so-called ``secondary flows",~\citep{prandtl}, see also the paper by \cite{einstein}, the initial condition leads to the following (schematic) picture in the $xz$--plane, in which we also indicate the point where a conceivable
finite-time singularity (or at least an extremely strong growth of vorticity) is observed numerically. The singularity was not predicted by the classics, who were mostly concerned with slightly viscous flows. In the presence of viscosity the scenario of singularity formation discussed here can be ruled out, due to well-known regularity criteria for the Navier-Stokes equations near boundaries, such as \cite{seregin} or \cite{gkt2006}.

\begin{center}
\includegraphics[width=100mm]{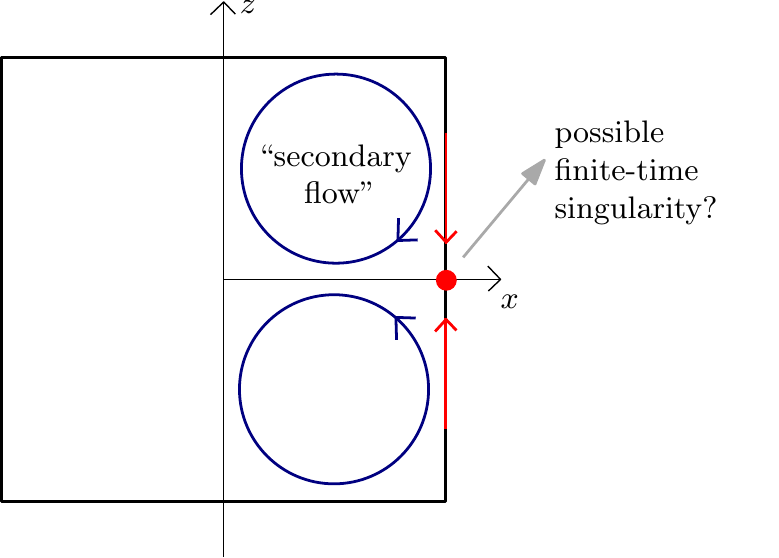}
\end{center}

One of the standard forms of the axi-symmetric Euler equations in the usual cylindrical coordinates $(r,\theta, z)$ is
\begin{subequations}\label{axi-sym}
\begin{align}
\left(\frac{\omega^\theta}{r}\right)_t+
u^r\left(\frac{\omega^\theta}{r}\right)_{r} + u^z\left(\frac{\omega^\theta}{r}\right)_{z} & =\left(\frac{(u^\theta)^2}{r^2}\right)_{z}\\
(ru^\theta)_{t}+u^r(ru^\theta)_{r}+u^z(ru^\theta)_{z} & = 0\,,
\end{align}
\end{subequations}
with the understanding that $u^r, u^z$ are given from $\om^\theta$ via the Biot-Savart law (which follows from the equations $\curl u =\om,\,{\rm div}\,u=0$, together with the boundary condition $u \cdot n=0$ and suitable decay at $\infty$ or, respectively, periodicity in $z$). See e.g. \cite{mb2002} for more details.
We will link these equations to the system~\rf{hl}.
\bb

A somewhat similar scenario can be considered for the 2d Boussinesq system in a half-space $\Om=\{(x,y)\in\R\times (0,\infty)\}$ (or in a flat half-cylinder $\Om={\bf S}^1\times (0,\infty)$), which we will write in the vorticity form:
\begin{subequations}\label{boussinesq}
\begin{align}
\om_t+u^x\om_x+u^y \om_y & = \theta_{x}\\
\theta_t+u^x \theta_x + u^y \theta_y & = 0\,.
\end{align}
\end{subequations}
 Here $u=(u^x,u^y)$ is obtained from $\omega$ by the usual Biot-Savart law (which follows from the equations $\curl u=\om, \,\,{\rm div}\, u =0\,$, the boundary condition $u\cdot n=0$ at $\partial\Omega$ and/or suitable decay at $\infty$), and $\theta$ represents the fluid temperature.

It is well-known that this system has properties similar to the axi-symmetric Euler, at least away from the rotation axis, with $\theta$ playing the role of the square of the swirl component $u^\theta$ (in standard cylindrical coordinates) of the velocity field $u$. For the purpose of the comparison with the axi-symmetric flow, the last picture should be rotated by $\pi/2$, after which it resembles the picture relevant for~\rf{boussinesq}:

\begin{center}
\includegraphics[width=90mm]{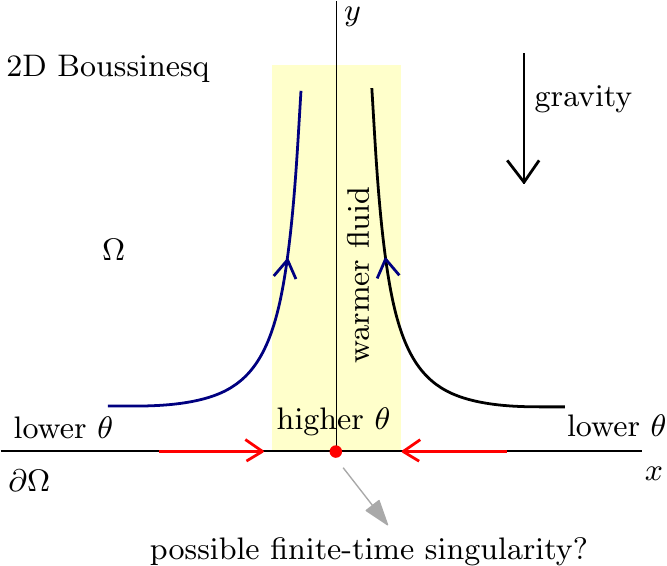}
\end{center}

We see that both in the 3d axi-symmetric case and the 2d Boussinesq the best chance for possible singularity formation seems to be at the points of symmetry at the boundary, and this is where we will focus our attention.

Note that~\rf{boussinesq} restricts well to the boundary $\{(x,y),\,y=0\}$, which can be identified with the real line.
At the boundary we can write with slight abuse of notation $u\sim (u(x),0)$, and we see that the system would close (as a system on the real line) if we could express the boundary velocity $u(x)$ via the restriction of $\om$ to the boundary.

Similar considerations apply to~\rf{axi-sym} \, when we restrict the equations to the boundary of $\Om$ given by, say, $r=1$. The restricted equations will look the same as the restricted Boussinesq equations, after the identifications $x\leftrightarrow z,\,\,\theta \leftrightarrow (u^\theta)^2,\,\,\omega\leftrightarrow\omega^\theta,\, u^x\leftrightarrow u^z$, with the understanding that the Biot-Savart is not quite identical, although the leading terms at the boundary are similar.

To obtain a closed 1d model, we need a way to model the Biot-Savart law. The most natural way to do so is probably to assume that $\om$ is (nearly) constant in $y$ close to the boundary, and discount the influence from the rest of the fluid. Let us assume the thickness of the layer where $\om$ is constant in $y$ is $a>0$. We adopt a convention that positive vorticity generates clock-wise rotation.\footnote{This is more convenient for our purposes here than the more usual convention with the opposite sign.} Under these assumptions it is easily seen that a reasonable 1d model of the Biot-Savart law (in the case when $\Om$ is the upper half-plane) is given by
\be\la{biot-sav}
u(x)=\int_{-\infty}^{\infty} \tilde k(x-y)\om(y)\,dy\,,
\ee
with the kernel $\tilde k$ determined by
\be\la{kernel1}
\tilde k(x_1)= \int_0^a \left. \frac{\partial}{\partial x_2}\right|_{x_2=0} G((x_1,x_2),(0,y_2))\,dy_2\,,
\ee
where
\be\la{greenf}
G(z,w) = \frac{1}{2\pi} \log |z-w|-\frac{1}{2\pi} \log |z-w^*|\,,\qquad w^*=(w_1,-w_2)\,,
\ee
is the Green function of the Laplacian in the upper half-plane. A simple calculation gives
\be\la{pre-kernel}
\tilde k(x)=\frac{1}{\pi} \log \frac{|x|}{\sqrt{x^2+a^2}}\,.
\ee
 One could work with this kernel, but we will simplify it to
\be\la{kernel}
k(x)=\frac{1}{\pi}\log |x|\,.
\ee
This kernel has the same singularity at $0$ and gives exactly the Biot-Savart law~\rf{bslaw1}. The somewhat unnatural behavior of $k$ for large $x$ will be alleviated by our symmetry assumptions.
We see that 1d models discussed in the previous section can be interpreted, to some degree, as the boundary dynamics for 2d flows, or 3d axi-symmetric flows with swirl. A similar (although somewhat less straightforward) calculation can be carried out in the axi-symmetric case, leading again to kernel~\rf{kernel}, as can be expected from the fact that the leading order terms in the corresponding elliptic operators are the same. \\ 

We acknowledge that the assumptions made in the above derivation of the model do not perfectly capture the situation near the boundary. For example, in the axi-symmetric case with the boundary at $r=1$, instead of being roughly constant in $r$ for $r$ slightly below $1$, the solution $\omega^\theta$ obtained from numerical simulations was observed to exhibit nontrivial variations \citep{lh2014a}. Nonetheless, according to the preliminary numerical evidence reported in \citet[][Section 5]{lh2014b}, the solution of the 1d model \eqref{hl} on ${\bf S}^1$ appears to develop a singularity in finite-time \emph{in much the same way} as the full simulation of the axi-symmetric flow. This supports the relevance of the 1d model for the finite-time blowup of the full 3D problem.

\subsection{Statement of the Main Results}\label{ssec_m_res}

The HL-model \eqref{hl} has the scaling invariance
\begin{equation}
  \omega(x,t) \to \omega(\lambda x,t),\qquad u(x,t) \to \frac{1}{\lambda}\, u(\lambda x,t),\qquad \theta(x,t) \to \frac{1}{\lambda}\, \theta(\lambda x,t),
  \label{invariance}
\end{equation}
and is invariant under the translation $\theta \to \theta + \text{const}$. This suggests that the critical space for the local well-posedness of the model should be
\begin{displaymath}
  (\theta_{0x},\omega_{0}) \in L^{\infty} \times L^{\infty}\qquad \text{or perhaps}\qquad (\theta_{0x},\omega_{0}) \in \dot H^{\frac{1}{2}} \times \dot H^{\frac{1}{2}}.
\end{displaymath}
It is not clear if the equation is locally well-posed in these spaces.\footnote{The answer might also depend on technical details of the definitions. In this context we refer the reader to the recent works of~\cite{bl2013, bl2014} on the ill-posedness of the Euler equation in critical spaces.} From the well-posedness proofs of the Euler equation in the slightly subcritical cases it should be clear that our system is locally well-posed for
\begin{displaymath}
  (\theta_{0x},\omega_{0}) \in H^{s}\times H^s\,,
\end{displaymath}
for any $s>\frac{1}{2}$ (see e.g. \cite{bl2013} for discussion and further references on subcritical local existence results).
We will not study these relatively standard issues here, as our interest is in the breakdown of smooth solutions.
Our main result is the following:

\begin{theorem}\label{thm_1d_blowup}
Both in the periodic case ($x\in {\bf S}^1$) and in the real-line case ($x\in \R$) with compactly supported initial data $(\theta_{0x},\om_0)$, one can find smooth initial conditions
such that the HL-model cannot have a smooth global solution starting from those data. 
\end{theorem}


Our proof of this statement is indirect, using an argument by contradiction, somewhat similar in spirit to the classical proofs of the blow-up in the non-linear Schr\"odinger equation \citep{glassey1977, talanov1971}, based on the virial identity. We do not get any detailed information about the nature of blow-up. For the periodic HL-model, the main quantity used in our proof is essentially
\be\la{quant1}
\int_0^{x_0} \frac{\theta(x,t)}{x}\,dx\,.
\ee
In view of the methods used in~\cite{ccf2005} the quantity
\be\la{quant11}
\int_0^{x_0} \frac{\om(x,t)}{x}\,dx
\ee
might look more natural, but we did not find a proof based on this quantity.
A key component of our proof are monotonicity properties of the Biot-Savart kernel, which are important in Lemma~\ref{l2} and Lemma~\ref{l4}.
Under some assumptions, a good quantity for a relatively simple proof of the blow-up for the CKY-model is an entropy-type integral defined in~\rf{e1}.

Some information about the nature of the blow-up can be obtained from Beale-Kato-Majda (BKM) type criteria~\citep{bkm1984}, which one can adapt to our case. For example, for a smooth solution (with appropriate decay, in case of $\R$) defined on a time interval $[0,T)$ any of the following conditions imply that the solution can be smoothly continued beyond $T$:
\be\la{bkmq}
\int_0^T ||u_x(\cdot,t)||_{L^\infty}\,dt < +\infty\,\,,\qquad
\int_0^T ||\theta_x(\cdot, t)||_{L^\infty}\,dt < +\infty\,.
\ee
See, for example, \cite{danchin2013} for the proof of an analogous result for the 2D inviscid Boussinesq system which can be adapted to our case in a straightforward way.

The BKM condition for 3d Euler is of course
\be\la{bkm3d}
\int||\om(\cdot,t)||_{L^\infty}\,dt<\infty\,,\qquad \om=(\om_1,\om_2,\om_3)=\curl u\,.
\ee
For the 3d axi-symmetric flow considered in \cite{lh2014a} this becomes
\be\la{bkmas}
\int_0^T \left( ||\om^\theta(\cdot,t)||_{L^\infty}+
||(u^\theta)_{z}(\cdot,t)||_{L^\infty}+
||\tfrac{1}{r} (ru^\theta)_{r}(\cdot, t)||_{L^\infty}\right)\,dt<\,+\infty\,,
\ee
which, when applied in the context of the 1d model is a stronger condition than either of the conditions in~\rf{bkmq}. It is natural to ask whether in the context of the HL-model
\be\la{bkmopen}
\int_0^T ||\om(\cdot, t)||_{L^\infty}\,dt<+\infty
\ee
 is a good BKM-type condition. It is possible to prove that the quantity $\int_0^T ||\om(\cdot, t)||_{L^\infty}\,dt$ cannot stay finite if the main quantity used in our proof becomes infinite,
 but one could conceivably have some loss of smoothness while both our main quantity in the proof and $\int_0^T ||\om(\cdot, t)||_{L^\infty}\,dt$ remain finite. This is unlikely, but, strictly speaking, our proofs do not rule that out.
 A similar question also seems to be open for the 2d Boussinesq system~\rf{boussinesq}: it is not clear whether the condition
 \be\la{bkmbouss}
 \int_0^T||\om(\cdot,t)||_{L^\infty}\,dt<\infty
 \ee
 is a good BKM-type condition for the system.

 Another natural question one can ask about the nature of blow up is whether the analog of the kinetic energy $\int_{{\bf S}^1}|u|^2\,dx$ stays finite at blow up time.
Note that for the 2D Boussinesq system \eqref{boussinesq} there is a natural energy
\begin{equation}\label{2dbenergy} E = \frac12 \int_{\Omega} |u(x,y)|^2\,dx dy - \int_{\Omega} y\theta(x,y)\,dx dy \end{equation}
which is conserved for reasonable classes of solutions. This conservation implies that the kinetic part of the energy should remain finite for the solutions of the 2D Boussinesq system
at the blow up time (if any). In the 1D model we do not expect an exact conservation of some quantity analogous to \eqref{2dbenergy} since it only models a boundary layer and not a closed system.
However, if the kinetic energy became infinite at blow up time, it would clearly indicate a divergence between the 1D model and the original equation. In addition, finite time loss
of regularity in solutions to Euler equations with infinite energy is well documented in the literature (see \cite{Const} for examples and further references).
In the Appendix, we will provide a simple argument showing that for a class of initial data leading to finite time blow up in the HL-model, the kinetic energy stays finite.
Moreover, more is true: any $L^p$ norm $\|u\|_{L^p}$ with $p < \infty$ and the BMO norm $\|u\|_{BMO}$ remain finite at blow up time. The key to these results
is the control of the $L^1$ norm of the vorticity. 

\section{The CKY and HL models on the real line}


In this section, we discuss the CKY and HL models on the real line. Our primary reason is to outline in the most transparent
setting the ideas that will be later used in the rigorous and more technical proof of the finite time blow up for the HL model in the periodic case.
The initial data for which we will prove the blow up will not be in the most general class. We will focus on the main underlying
structure and ideas that will be fully developed for the periodic case in a later section - but will also look more technical there.
Our main goal is the proof of finite time blow up for the HL-model in the periodic setting, and this section can be thought of as a preview of the ideas and connections employed in this proof in a situation without technical distractions which one has to deal with in the periodic case.

To start with, let us study more carefully the Biot-Savart laws for both models.
 Motivated by the structure of the Biot-Savart laws, a change of variable is introduced so that the velocities for both models become convolutions. Using the new variables, we will
 prove the finite time blow up of the CKY model using an ``entropy'' functional, and then prove the finite time blow up for the HL model using another natural functional.

\subsection{Comparison of velocities}
\label{comparison}
Let us first look at the velocity field $u$ in the HL model and the CKY model respectively. \\

For the HL model on $\R$, we use the the following representations
$$
 u_{\text{HL}}(x):=\frac{1}{\pi}\int_\R \omega(y)\log|x-y|dy
\quad \mbox{ and }\quad
\partial_xu_{\text{HL}}(x)=P.V.\frac{1}{\pi}\int_\R\frac{\omega(y)}{x-y}dy=(H\omega)(x).$$
 We will study the situation
when $\om$ is odd and $\theta$ is even, i.\ e.\
\be\la{b4}
\om(-x,t)=-\om(x,t),\quad \theta(-x,t)=\theta(x,t)\,.
\ee
In this situation one can restrict the attention to $x\in(0,\infty)$,
and the expression for $u$ in terms of $\om$ can be written as
\be\la{b5}
\frac{u_{\text{HL}}(x)}{x}=-\frac{1}{\pi}\int_0^\infty \frac{y}{x}\log\left |\frac{x+y}{x-y}\right|\,\,\om(y)\,\frac{dy}{y}\,.
\ee
We note that the last integral is of the form
\be\la{b6}
\int_0^\infty M\left(\frac{x}{y}\right) \om(y) \,\frac{dy}{y}\,,
\ee
which represents convolution in the multiplicative group $\R_+$ taken with respect to the natural invariant measure $\frac{dy}{y}$.
The kernel $M$ is given by
\be
M(s)= \frac{1}{s}\log \left| \frac{s+1}{s-1}\right |\,\,,\qquad s>0\,.
\ee
We will use the following decomposition of $M$:
\be\la{b7}
M(s)= \pul\left(\frac{1}{s}+s\right)\log\left|\frac{s+1}{s-1}\right|+
\pul\left(\frac{1}{s}-s\right)\log\left|\frac{s+1}{s-1}\right|=M_{\rm sym}(s)+M_a(s)\,.
\ee
We have
\be\la{b8}
M_{\rm sym}\left(\frac{1}{s}\right)=M_{\rm sym}(s), \qquad
M_{a}\left(\frac{1}{s}\right)=-M_a(s)\,.
\ee
We collect some properties of the function $M$ in the following lemma. The proof is elementary and will be omitted here; for the periodic case the corresponding lemma is
a more technical Lemma \ref{lmm_kernel}, for which we will provide a complete proof.
\begin{lemma}\label{l1} The function $M$ has the following properties:
\begin{itemize}
\item[(i)] $M$ is increasing on $(0,1)$ and decreasing on $(1,\infty)$.
\item[(ii)] $\lim _{s\to 0_+} M(s)=2\,,\quad \lim_{s\to 0_+}M'(s)=0$\,.
\item[(iii)] $M_a$ is continuous and decreasing in $(0,\infty)$,
with $\lim_{s\to 0_+}M_a(s)=1$.
\item[(iv)] $M(s)=\frac{2}{s^2}+O\left(\frac{1}{s^3}\right)\,,\quad s\to\infty.$
\end{itemize}
\end{lemma}

\sm
Note that the velocity of the CKY model can be written in a form similar to \eqref{b6} (if we assume $\omega$ is
compactly supported in $(0,1)$). The Biot-Savart law here is given by
\be\la{b10}
\frac{u_{\text{CKY}}(x)}{x}=-\frac{1}{\pi} \int_0^\infty \tilde M\left(\frac{x}{y}\right)\,\om(y) \frac{dy}{y}\,,
\ee
with
\be\la{b9}
\tilde M(s) =2\chi_{[0,1]}(s)\,,
\ee
where we use the notation $\chi_A$ for the characteristic function of the set $A$. In some sense, we can think of $\tilde M$ as the most natural crudest approximation of $M$.
Moreover, note that the velocity for the HL model is ``stronger'' than for the CKY model in the following sense:
\be\la{b11}
0\le \tilde M \le M,\\\quad \int_0^\infty (M(s)-\tilde M(s))\frac{ds}{s} <\infty\,.
\ee
As a result, when $\omega(y)\geq 0$ for $y\geq 0$, we have $u_{\text{HL}}(x) \leq u_{\text{CKY}} \leq 0$.
This suggests that the finite time blow up for the CKY model should imply the finite time blow up for the HL model.
 However, there is no straightforward comparison principle that would directly substantiate such a claim. The Biot-Savart law
of the CKY model is simpler and can be made local after differentiation. This makes the CKY model easier to deal with.
For the HL model more sophisticated arguments are needed, due to the combination of its non-local and non-linear nature.


\subsection{A change of coordinates}
In view of formulae from the last subsection and the scaling symmetry~\rf{invariance}, it seems natural to work with the variables $\xi, U(\xi), \Om(\xi),\Theta(\xi)$
defined by
\be\la{n1}
x=e^{-\xi}\,,\quad U(\xi)=-\frac{u(x)}{x}\,,\quad \Omega( \xi)=\om(x)\,,\quad \Theta(\xi)=-\theta(x)+\theta(0)\,.
\ee
In addition, we will use the notation
\be\la{n2}
\rho(\xi)=\Theta_\xi(\xi)\,.
\ee
We note that
\be\la{n2b}
\Theta(\xi)=-\int_\xi^\infty \rho(\eta)\,d\eta\,.
\ee
In the coordinates~\rf{n1} the HL model and the CKY model can both be written as
\be\la{n3}
\begin{array}{rcl}
\Om_t+U\Om_\xi & = & e^\xi \Theta_\xi\,, \\
\Theta_t+U\Theta_\xi & = & 0\,,
\end{array}
\ee
where $U$ is given by $U_{\text{HL}}$ and $U_{\text{CKY}}$ respectively. For the HL model, its Biot-Savart law \eqref{bslaw1}
becomes \footnote{We use the usual notation $f*g$ for convolution.}
\be\la{n4}
U_{\text{HL}}=K*\Om\,,
\ee
with
\be\la{n5}
K(\xi)=\frac{1}{\pi}\,M(e^{-\xi})\,.
\ee
 By Lemma~\ref{l1} the function $K$ is increasing on $(-\infty,0)$ and is decreasing on $(0,\infty)$ with \\$\lim_{\xi\to\infty} K(\xi)=\frac{2}{\pi}$.
 For the CKY-model, the Biot-Savart law becomes
 \be\la{n6}
 U_{\text{CKY}}=\tilde K*\Om\,,
 \ee
where
\be\la{n7}
\tilde K= \frac{2}{\pi}\chi_{[0,\infty)}\,.
\ee
Hence in the CKY-model we have
\be\la{n8}
U_{\text{CKY}}(\xi)=\frac{2}{\pi}\int_{-\infty}^\xi \Om(\xi')\,d\xi'\,\,.
\ee
The decomposition~\rf{b7} corresponds in the new coordinates to
\be\la{n9}
K(\xi)=K_{\rm sym}(\xi)+K_a(\xi)\,,
\ee
with
\be\la{n10}
K_{\rm sym}(\xi)=\pul\left(K(\xi)+K(-\xi)\right),\qquad K_a(\xi)=\pul\left(K(\xi)-K(-\xi)\right)\,.
\ee
The function $K_a(\xi)$ is continious and increasing, with limits $\pm \frac{1}{\pi}$ at $\pm\infty$, respectively. The function $K_{\rm sym}$ is everywhere above $\frac{1}{\pi}$, with the difference
$K_{\rm sym} -\frac{1}{\pi}$ being integrable.

\subsection{Blow up in the CKY-model -- the entropy functional approach}

Let us assume that
\be\la{a1}
\hbox{$\rho_0(\xi)=\Theta_\xi(\xi,0)$ is non-negative, compactly supported, with $\int_\R\rho_0(\xi)\,d\xi=1$\,.}
\ee
In fact, we will assume (without loss of generality) that the support of $\rho$ is in $(0,\infty)$.
This can be achieved for many compactly supported initial data $\theta_0(x)$, it just needs to be non-decreasing,
to be constant near zero, and to vanish for $x \geq 1.$ Observe that $\Theta_0(\xi)$ in this case is non-decreasing
and approaches zero as $\xi \rightarrow \infty,$ so in particular $\Theta_0(\xi) \leq 0$ for all $\xi.$
The normalization $\int_\R\rho_0\,d\xi=1$ is
not important, but the condition $\rho_{0}\ge 0$ will play an important role in our argument.
It can be relaxed to the condition that $\rho_0(\xi)$ has a ``bump" of this form, but we
will not pursue this here.

We also assume that the initial
data $\Omega_0(\xi)$ is non-negative, compactly supported.

The evolution equation for $\rho$ is
\be\la{a2}
\rho_t+(U\rho)_\xi=0\,.
\ee
 This is the usual equation of continuity and with our assumptions this means that $\rho$ can be considered as a density transported by $U$.

 Let us consider the entropy of $\rho$, defined by
 \be\la{e1}
 I=I(t)=\int_\R -\log\rho\,\,\rho\,d\xi\,.
 \ee
In this subsection, we use an approach different from \cite{cky2015} to show the finite time blow up of the CKY model. Namely, we will prove that the entropy functional $I(t)$ blows up in finite time, implying that the solution must have a finite time singularity as well. 

 \begin{lemma}\label{l3} Assume that a smooth positive density $\rho$ (of unit total mass) is
 compactly supported in $[0,\infty)$. Then
 \be\la{e2}
 \int_\R\xi\,\rho(\xi)\,d\xi\ge \exp\,(I-1)\,.
 \ee
 \end{lemma}
\begin{proof}
 Define $f(\cdot)$ by $\rho(\xi)=f(\xi)e^{-\xi/J}$ where $J:= \int_\R\xi\,\rho(\xi)\,d\xi>0$.
 Observe
 $$
 I= \int_\R -\log f\,\,\rho\,d\xi\, +\frac{1}{J}\int_\R \xi\,\rho\,d\xi\,
  =\int_\R -\log f\,\,\rho\,d\xi\, + 1.
 $$
 So, by Jensen's inequality, we obtain
 $$e^{I-1}
 =\exp \biggl( \int_\R -\log f\,\,\rho\,d\xi \biggr)
 \leq \int_\R \,\frac{\rho}{f}\,d\xi\,= J.
 $$
 \vspace{-0.4cm}
\end{proof}

 When $\rho$ evolves by the equation of continuity~\rf{a2}, we have
 \be\la{e3}
 \frac{d}{dt} I = \int_\R U_\xi\rho \,d\xi\,.
 \ee
 In the CKY-model we have $U_\xi=\frac{2}{\pi}\Om$, and hence
 \be\la{e4}
 \frac{d}{dt} I =\frac{2}{\pi}\int_\R\Om\rho\,d\xi\,.
 \ee
 The equation
 \be\la{e5}
 \Om_t+U\Om_{\xi}=e^\xi\rho\,,
 \ee
together with the equation of continuity for $\rho$ give
\be\la{e6}
\frac{d}{dt}\int \Om\rho = \int e^\xi \rho\rho\,d\xi=\int e^{\xi+\log\rho}\rho\,d\xi.
\ee
By Jensen's inequality the last integral is greater or equal to
\be\la{e7}
\exp\left(\int \xi\rho\,d\xi-I\right)\,.
\ee
Estimating the integral in the last expression from below by Lemma~\ref{l3}, we finally obtain:
\be\la{e8}
\ddot I \ge \frac{2}{\pi} e^{(e^{I-1}-I)}\,.
\ee
In addition, the above calculations show that $I$ is strictly increasing and convex as a function of $t$. It is now easy to see that $I$ must always become infinite in finite time.

In the original coordinates, we have
\[ I(t) = \int_0^\infty \theta_x \log(x \theta_x)\,dx. \]
It is not difficult to see, using the blow up criteria \eqref{bkmq}, that $I(t)$ can only become infinite if the solution loses regularity.
We will sketch this argument below in the next section when proving Theorem~\ref{thm_1d_blowup}.

We have proved the following result:

\begin{theorem}
For the initial data as above, the CKY-model given by \eqref{n3} with Biot-Savart law \eqref{n6} develops a singularity in finite time.
\end{theorem}

\noindent
{\bf Remark.} Due to the local nature of equation $U_\xi=\Om$ in the CKY model, it is not hard to see that a similar argument applies whenever $\theta$ is increasing on some interval
and the initial vorticity $\om$ is non-negative on $(0,\infty)$.

\subsection{Blow up in the CKY-model and the HL-model -- a new functional}

If one tries to reproduce the above proof for the HL model, it quickly becomes apparent that the condition $\rho_\xi \geq 0$ seems necessary for the proof to go through. This leads to growth at infinity
for the initial data, which is not very reasonable.
Hence the above proof cannot be directly applied to the HL model. However, a different functional $\int_\R \rho \xi d\xi$ can be used to prove finite time blow up for both models, which we will demonstrate below.
Let's assume $\rho_0$ and $\Omega_0$ satisfy the conditions in the previous subsection.
Let $F(t) := \int \rho \xi d\xi = -\int_0^\infty \Theta d\xi$. Taking the time derivative of $F$, we have
\[
\frac{dF}{dt} = \int U \Theta_\xi \,d\xi
\]
for both models. For the CKY model, since $(U_{CKY})_\xi = \frac{2}{\pi} \Omega$, the right hand side is equal to $$G(t) := -\frac{2}{\pi} \int \Omega \Theta d\xi.$$ For the HL model, since its velocity field is ``stronger'' than in the CKY model (see subsection \ref{comparison}), we have
$\int U_{HL} \Theta_\xi d\xi \geq \int U_{CKY} \Theta_\xi d\xi = G(t)$.

Now let us take the time derivative of $G(t)$:
\[
\begin{split}
\frac{dG}{dt} &= -\frac{2}{\pi}\int \Omega_t \Theta + \Omega \Theta_t d\xi\\
&= \frac{2}{\pi} \int U\Omega_\xi \Theta + U\Omega \Theta_\xi - e^\xi \Theta \Theta_\xi d\xi\\
&= -\frac{2}{\pi} \int U_\xi \Omega \Theta d\xi + \frac{1}{\pi} \int e^\xi \Theta^2 d\xi
\end{split}
\]
For the CKY model, the first term on the right hand side is positive,
since $U_\xi = \frac{2}{\pi} \Omega$ and $\Theta \leq 0$. For the HL model,
we claim that $\int_{-\infty}^\xi U_\xi(\eta)\Om(\eta)\,d\eta \geq 0$
for any $\Omega\geq 0$ and for any $\xi\in\R$. The proof of the claim will be given in the next subsection.
Once this is proved, due to the fact that $\Theta\leq 0$ and is increasing, we can use integration by parts and the substitution $\Theta(\xi)=s$ to obtain
the following estimate for the HL model:
\[
-\frac{2}{\pi} \int U_\xi \Omega \Theta d\xi = \frac{2}{\pi}\int_{\Theta_{min}}^0 \int_{-\infty}^{\Theta^{-1}(s)} U_\xi \Omega \,d\xi\, ds \geq 0,
\]
where $\Theta^{-1}$ is the inverse function of $\Theta$. As a result, for both models, we have
\[
\begin{split}
\frac{dG}{dt} \geq \frac{1}{\pi} \int e^\xi \Theta^2 d\xi \geq \frac{1}{\pi}\Big(\int_0^\infty
-\Theta e^{\frac{\xi}{2}}e^{-\frac{\xi}{2}}\,d\xi\Big)^2 = \frac{1}{\pi} F(t)^2,
\end{split}
\]
which gives us the system
\[
\frac{dF}{dt} \geq G(t) \geq 0, \quad\quad
\frac{dG}{dt} \geq \frac{1}{\pi} F(t)^2,
\]
and one can obtain that $F(t)$ blows up in finite time by a standard ODE argument.

\subsection{Quadratic forms in the HL-model}
Now we prove the claim in the last subsection. Let
\be\la{n11}
I(\Om,\xi)=\int_{-\infty}^\xi U_\xi(\eta)\Om(\eta)\,d\eta\,.
\ee
\begin{lemma} \label{l2} For any smooth compactly supported $\Om\ge 0$ we have $I(\Om,\xi)\ge 0$ for all $\xi$.
\end{lemma}
\begin{proof}
\sm
Let us write $\Om=\Om_l+\Om_r$, where
\be\la{p1}
\Om_l=\Om\,\chi_{(-\infty,\xi]}\,,\qquad \Om_r=\Om\,\chi_{(\xi,\infty)}\,.
\ee
We have
\be\la{p2}
U=U_l+U_r,\quad U_l=K*\Om_l,\quad U_r=K*\Om_r\,,
\ee
and
\be\la{p3}
I=I(\Om,\xi)=\int_\R U_\xi(\eta)\Om_l(\eta)\,d\eta\,=\int_R U_{l\xi}\Om_l\,d\eta+\int_R U_{r\xi}\Om_l\,d\eta\,.
\ee
We claim that in the last expression both integrals are non-negative. Denoting by $K'$ the derivative of $K$ (taken in the sense of distributions), we can write for the first integral \be\la{p4}
\int_\R U_{l\xi}\Om_l\,d\eta=\int_R\int_\R K'(\eta-\zeta)\Om_l(\zeta)\Om_l(\eta)\,d\eta\,d\zeta=\int_R\int_\R K_a'(\eta-\zeta)\Om_l(\zeta)\Om_l(\eta)\,d\eta\,d\zeta\ge 0\,,
\ee
as $K_a$ is increasing.
The second integral is equal to
\be\la{p5}
\int_\R\int_\R K'(\eta-\zeta)\Om_r(\zeta)\Om_l(\eta)\,d\,\zeta\,d\eta,
\ee
and we note that the integration can be restricted to the domain $\{\eta<\zeta\}$, as the integrand vanishes elsewhere.
As $K'(\xi)>0$ for $\xi<0$, the result follows.
\end{proof}

\section{The Finite Time Blowup for the HL-Model in the periodic setting}\label{sec_blowup}
In this section we consider the HL model in $[0,L]$ with periodic condition, and we will
prove our main result - Theorem~\ref{thm_1d_blowup} in the periodic setting.
Now $u_x$ is the (periodic) Hilbert transform of $\omega$, namely
\begin{displaymath}
  u_{x}(x) = \frac{1}{L} \PV\int_{0}^{L} \omega(y) \cot \bigl[ \mu (x-y) \bigr]\,dy =: H\omega(x),\qquad \mu := \pi/L.
\end{displaymath}
As a result, the nonlocal velocity $u$ is defined by:
\begin{equation}
  u(x) = Q\omega(x) := \frac{1}{\pi} \int_{0}^{L} \omega(y) \log \bigl| \sin[\mu(x-y)] \bigr|\,dy\,.
  \label{eqn_m_hl_u}
\end{equation}

In this case the change of variables of the previous section is less natural, and we will work in the original coordinates.

In this section, we consider smooth odd periodic initial data $\theta_{0x},\, \omega_{0}$ with period $L$.
We note that such functions are also odd with respect to the origin taken at $x = \frac{1}{2} L$.
 In addition,
we suppose $\theta_{0x},\, \omega_{0} \geq 0$ on $[0,\frac{1}{2} L]$. Then, the initial velocity $u_0$
also becomes odd at $x=0$ and $\frac{1}{2} L$ and $u_{0} \leq 0$ on $[0,\frac{1}{2} L]$ (for the proof of the last
assertion on $u_0$, see \eqref{eqn_u}).
Lastly, since \eqref{hl} is invariant under a constant shift $\theta \to \theta + \text{const}$,
 we may assume $\theta_{0}(0) = 0$.

Thanks to transport structure of \eqref{hl}, the evolution preserves
the assumptions as long as the solution exists. That is, $\theta_{x},\, \omega, u$ stay odd at $x=0$ and $\frac{1}{2} L$
with period $L$, $\theta(0) = 0$, and $\theta_{x},\, \omega, (-u) \geq 0$ on $[0,\frac{1}{2} L]$.


The proof will be based on an idea which builds on the insight gained from the whole line case. It considers the integral
\begin{displaymath}
  I(t) := \int_{0}^{L/2} \theta(x,t) \cot(\mu x)\,dx
\end{displaymath}
and shows that it must blow up in finite time if $I(0)>0$. As will be shown below, this implies the blowup of the HL-model.

The proof relies on the following two lemmas, which reveal the key properties of the velocity $u$ as defined by the Biot-Savart law \eqref{eqn_m_hl_u}.
These Lemmas can be viewed as analogs of the Lemmas~\ref{l1} and \ref{l2} from the whole line case.
It is worth noting that both lemmas follow directly from \eqref{eqn_m_hl_u} and do not depend on the actual dynamics of the flow.

\begin{lemma}\label{l3a}
Let $\omega$ be periodic with period $L$ and odd at $x = 0$ and let $u = Q\omega$ be as defined by \eqref{eqn_m_hl_u}. Then for any $x \in [0,\frac{1}{2} L]$,
\begin{equation}
  u(x) \cot(\mu x) = -\frac{1}{\pi} \int_{0}^{L/2} K(x,y) \omega(y) \cot(\mu y)\,dy,
  \label{eqn_u_ker}
\end{equation}
where
\begin{equation}
  K(x,y) = s \log \biggl| \frac{s+1}{s-1} \biggr|\qquad \text{with}\qquad s = s(x,y) = \frac{\tan(\mu y)}{\tan(\mu x)}.
  \label{eqn_K}
\end{equation}
Furthermore, the kernel $K(x,y)$ has the following properties:
\begin{enumerate}
  \item $K(x,y) \geq 0$ for all $x,\, y \in (0,\frac{1}{2} L)$ with $x \neq y$;
  \vspace{1mm}
  \item $K(x,y) \geq 2$ and $K_{x}(x,y) \geq 0$ for all $0 < x < y < \frac{1}{2} L$;
  \vspace{1mm}
  \item $K(x,y) \geq 2s^{2}$ and $K_{x}(x,y) \leq 0$ for all $0 < y < x < \frac{1}{2} L$.
\end{enumerate}
\label{lmm_kernel}
\end{lemma}

\begin{proof}
The velocity $u = Q\omega$ as defined by \eqref{eqn_m_hl_u} admits the following representation:
\begin{align}
  u(x) & = \frac{1}{\pi} \biggl[ \int_{0}^{L/2} + \int_{L/2}^{L} \biggr] \omega(y) \log \bigl| \sin[\mu(x-y)] \bigr|\,dy \nonumber \\
  & = \frac{1}{\pi} \int_{0}^{L/2} \omega(y) \Bigl\{ \log \bigl| \sin[\mu(x-y)] \bigr| - \log \bigl| \sin[\mu(x+y)] \bigr| \Bigr\}\,dy \nonumber \\
  & = \frac{1}{\pi} \int_{0}^{L/2} \omega(y) \log \biggl| \frac{\tan(\mu x) - \tan(\mu y)}{\tan(\mu x) + \tan(\mu y)} \biggr|\,dy. \label{eqn_u}
\end{align}
This shows that
\begin{displaymath}
  u(x) \cot(\mu x) = -\frac{1}{\pi} \int_{0}^{L/2} K(x,y) \omega(y) \cot(\mu y)\,dy,
\end{displaymath}
where $K(x,y)$ is the kernel defined by \eqref{eqn_K}, hence the first part of the lemma follows.

It remains to check that $K(x,y)$ satisfies the properties (a)--(c). To prove (a), note that $x,\, y \in (0,\frac{1}{2} L)$ implies $s(x,y) > 0$, hence $\abs{s+1} > \abs{s-1}$. The claim follows easily.

For (b), observe that $s(x,y) > 1$ for $0 < x < y < \frac{1}{2} L$. Hence we can express $K(x,y)$ as a Taylor series in terms of $s^{-1}$:
\begin{displaymath}
  K(x,y) = s \Bigl\{ \log \bigl( 1 + s^{-1} \bigr) - \log \bigl( 1 - s^{-1} \bigr) \Bigr\} = 2\sum_{n=0}^{\infty} \frac{s^{-2n}}{2n+1} \geq 2.
\end{displaymath}
By taking derivatives of $x$ on both sides and observing that $s_{x}(x,y) \leq 0$, we also deduce $K_{x}(x,y) \geq 0$ for $0 < x < y < \frac{1}{2} L$. This establishes (b).

Finally, to prove (c), we proceed as in (b) and note that $0 < s(x,y) < 1$ for $0 < y < x < \frac{1}{2} L$. Thus
\begin{displaymath}
  K(x,y) = s \Bigl\{ \log(1+s) - \log(1-s) \Bigr\} = 2\sum_{n=0}^{\infty} \frac{s^{2n+2}}{2n+1} \geq 2s^{2}.
\end{displaymath}
In addition, taking derivatives in $x$ and using $s_{x}(x,y) \leq 0$ shows $K_{x}(x,y) \leq 0$ for $0 < y < x < \frac{1}{2} L$. This completes the proof of the lemma.
\end{proof}

{\bf Remark.}
Lemma \ref{lmm_kernel} implies that the periodic HL-velocity $u_{\text{HL}}$ as defined by \eqref{eqn_m_hl_u} is, up to a constant factor, ``stronger'' than the CKY-velocity $u_{\text{CKY}}$ provided that $\omega \geq 0$ on $[0,\frac{1}{2} L]$. Indeed, since $\cot(\mu x) \sim (\mu x)^{-1}$ for small $x$, the properties (a) and (b) of the kernel $K(x,y)$ imply that, for any $z \in (0,\frac{1}{2} L)$, there exists a positive constant $C$ depending on $L$ and $z$ such that
\begin{displaymath}
  u_{\text{HL}}(x) \leq -Cx \int_{x}^{z} \frac{1}{y}\,\omega(y)\,dy,\qquad \forall x \in [0,z].
\end{displaymath}

The second lemma that we shall prove concerns the positivity of a certain quadratic form of the vorticity $\omega$, which plays the central role in the proof of Theorem \ref{thm_1d_blowup}.

\begin{lemma}\label{l4}
Let the assumptions in Lemma \ref{lmm_kernel} be satisfied and assume in addition that $\omega \geq 0$ on $[0,\frac{1}{2} L]$. Then for any $a \in [0,\frac{1}{2} L]$,
\begin{equation}
  \int_{a}^{L/2} \omega(x) \bigl[ u(x) \cot(\mu x) \bigr]_{x}\,dx \geq 0.
  \label{eqn_wvy_pos}
\end{equation}
\label{lmm_wvy_pos}
\end{lemma}

\begin{proof}
First, we note that for any $x,\, a \in [0,\frac{1}{2} L]$, $\omega(x)$ can be decomposed as
\begin{displaymath}
  \omega(x) = \omega(x) \chi_{[0,a]}(x) + \omega(x) \chi_{[a,\frac{1}{2} L]}(x) =: \omega_{l}(x) + \omega_{r}(x),
\end{displaymath}
where $\chi_{A}$ denotes the characteristic function of the set $A$. Using this decomposition and Lemma \ref{lmm_kernel}, we split the integral in \eqref{eqn_wvy_pos} as:
\begin{multline*}
  \int_{a}^{L/2} \omega(x) \bigl[ u(x) \cot(\mu x) \bigr]_{x}\,dx = -\frac{1}{\pi} \int_{0}^{L/2} \omega_{r}(x) \int_{0}^{L/2} \omega_{l}(y) \cot(\mu y) K_{x}(x,y)\,dy\,dx \\
  - \frac{1}{\pi} \int_{0}^{L/2} \omega_{r}(x) \int_{0}^{L/2} \omega_{r}(y) \cot(\mu y) K_{x}(x,y)\,dy\,dx =: -I_{1} - I_{2}.
\end{multline*}
Clearly, the lemma follows if both $I_{1}$ and $I_{2}$ are non-positive. For $I_{1}$, we observe from the definition of $\omega_{l}$ and $\omega_{r}$ that the integrand in $I_{1}$ is nonzero only when $y \leq x$. Using $K_{x} \leq 0$ as proved in Lemma \ref{lmm_kernel}(c) and the assumption that $\omega \geq 0$ on $[0,\frac{1}{2} L]$, we then deduce $I_{1} \leq 0$. As for $I_{2}$, we write $G := K_{x}$ and observe that
\begin{equation}
  I_{2} = \frac{1}{2\pi} \int_{0}^{L/2} \int_{0}^{L/2} \omega_{r}(x) \omega_{r}(y) T(x,y)\,dy\,dx,
  \label{eqn_ww_s}
\end{equation}
where $T(x,y) = \cot(\mu y) G(x,y) + \cot(\mu x) G (y,x)$. We shall show that $T(x,y) \leq 0$ for all $x,\, y \in (0, \frac{1}{2} L)$, which then implies $I_{2} \leq 0$. To this end, we compute
\begin{displaymath}
  G(x,y) = -\mu \csc^{2}(\mu x) \tan(\mu y) \biggl\{ \log \biggl| \frac{s+1}{s-1} \biggr| - \frac{2s}{s^{2}-1} \biggr\},
\end{displaymath}
and then
\begin{displaymath}
  T(x,y) = -\mu \bigl[ \csc^{2}(\mu x) + \csc^{2}(\mu y) \bigr] \log \biggl| \frac{s+1}{s-1} \biggr| + \mu \bigl[ \csc^{2}(\mu x) - \csc^{2}(\mu y) \bigr] \frac{2s}{s^{2}-1}.
\end{displaymath}
Thanks to Lemma \ref{lmm_kernel}(b) and (c), we have
\begin{displaymath}
  K(x,y) \geq \frac{2s^{2}}{s^{2}+1},\qquad \forall s \geq 0,
\end{displaymath}
which implies that
\begin{displaymath}
  \log \biggl| \frac{s+1}{s-1} \biggr| \geq \frac{2s}{s^{2}+1},\qquad \forall s \geq 0.
\end{displaymath}
It then follows that
\begin{displaymath}
  T(x,y) \leq -\biggl\{ \frac{4\mu s \csc^{2}(\mu x) \sec^{2}(\mu y)}{s^{2}+1} \biggr\} \cdot \biggl\{ \frac{\cos^{2}(\mu x) - \cos^{2}(\mu y)}{s^{2}-1} \biggr\} \leq 0,\qquad \forall x,\, y \in (0,\tfrac{1}{2} L),
\end{displaymath}
and hence $I_{2} \leq 0$. This completes the proof of the lemma.
\end{proof}

\begin{remark}
The positivity of the integral in \eqref{eqn_wvy_pos} corresponds to the positivity of $\int_{a}^{1} \omega(x) [u(x)/x]_{x}\,dx$ in the CKY-model, and so to Lemma~\ref{l2} in the whole line case.
Indeed, thanks to the local nature of the Biot-Savart law for the CKY model, the positivity of the latter integral follows almost immediately from the definition of the CKY-velocity $u_{\text{CKY}}$:
\begin{displaymath}
  \int_{a}^{1} \omega(x) \bigl[ u_{\text{CKY}}(x)/x \bigr]_{x}\,dx = \int_{a}^{1} \frac{1}{x}\, \omega^{2}(x)\,dx \geq 0.
\end{displaymath}
In this sense, Lemma \ref{lmm_wvy_pos} is quite surprising since the HL-velocity $u_{\text{HL}}$ as defined by \eqref{eqn_m_hl_u} does not have such a simple structure. It is the careful analysis of the kernel $K(x,y)$ (Lemma \ref{lmm_kernel}) and the symmetrization technique \eqref{eqn_ww_s} that make the estimate \eqref{eqn_wvy_pos} possible.
\end{remark}

\begin{proof}[Proof of Theorem \ref{thm_1d_blowup}]
We are going to
show a finite time blow up of the quantity \begin{equation}
  I(t) := \int_{0}^{L/2} \theta(x,t) \cot(\mu x)\,dx.
  \label{eqn_def_I}
\end{equation}
  We claim that this implies the finite-time blowup of the corresponding solution $(\theta,\omega)$ of the HL-model \eqref{hl}. Indeed, applying integration by parts to \eqref{eqn_def_I}, we see
\begin{displaymath}
  I(t) = -\frac{1}{\mu} \int_{0}^{L/2} \theta_{x}(x,t) \log\abs{\sin(\mu x)}\,dx,
\end{displaymath}
which can be bounded as follows for some constant $C > 0$:
\begin{displaymath}
  \abs{I(t)} \leq C \norm{\theta_{x}(\cdot,t)}_{L^{\infty}} \leq C \norm{\theta_{0x}}_{L^{\infty}} \exp \biggl\{ \int_{0}^{t} \norm{u_x(\cdot,s)}_{L^{\infty}}\,ds \biggr\}.
\end{displaymath}
Thus if $I(t)$ blows up at a finite time $T$, the same must hold true for $\int_{0}^{t} \norm{u_x(\cdot,s)}_{L^{\infty}}\,ds$. The finite-time blowup of the HL-model then follows from the corresponding Beale-Kato-Majda type condition \eqref{bkmq} for the HL model.

To prove the finite time blowup of $I(t)$, we assume $I(0)>0$ and consider
\begin{align*}
  \frac{d}{dt} I(t) & = -\int_{0}^{L/2} u(x) \theta_{x}(x) \cot(\mu x)\,dx \\
  & = \frac{1}{\pi} \int_{0}^{L/2} \theta_{x}(x) \int_{0}^{L/2} \omega(y) \cot(\mu y) K(x,y)\,dy\,dx,
\end{align*}
where in the second step we have used the representation formula \eqref{eqn_u_ker} from Lemma \ref{lmm_kernel}. According to our choice of the initial data and the properties of the kernel $K(x,y)$ as proved in Lemma \ref{lmm_kernel}(a)--(b), we have $\theta_{x},\, \omega \geq 0$ on $[0,\frac{1}{2} L]$, $K \geq 0$ for $y < x$, and $K \geq 2$ for $y > x$. Thus
\begin{displaymath}
\begin{split}
  \frac{d}{dt} I(t) &\geq \frac{2}{\pi} \int_{0}^{L/2} \theta_{x}(x) \int_{x}^{L/2} \omega(y) \cot(\mu y)\,dy\,dx \\
  &=\frac{2}{\pi} \int_0^{L/2} \theta(x) \omega(x) \cot(\mu x) dx=: J(t).
\end{split}
\end{displaymath}

Taking time derivative of $J(t)$ gives
\[
\begin{split}
\frac{d}{dt}J(t) &= \frac{2}{\pi} \int_0^{L/2} -\big(\theta(x)\omega(x)\big)_x u(x)\cot(\mu x) + \theta_x(x) \theta(x) \cot(\mu x)\,dx \\
&= \frac{2}{\pi} \int_0^{L/2} \theta(x) \omega(x) \big(u(x)\cot(\mu x)\big)_x dx + \frac{\mu}{\pi} \int_{0}^{L/2} \theta^{2}(x)\csc^{2}(\mu x)\,dx =: T_1 + T_2
\end{split}
\]
For $T_1$, since $\theta(x)$ is a nonnegative increasing function on $[0,L/2]$, one has, after integrating by parts
\begin{equation}\begin{split}
T_1 &
=
\frac{2}{\pi} \int_0^{L/2} \theta_y(y) \Big[\int_y^{L/2}\omega(x) \big(u(x)\cot(\mu x)\big)_xdx \Big] dy\geq0\,,
\end{split}\end{equation}

where
we applied Lemma \ref{lmm_wvy_pos} to get the inequality.

For $T_2$, we can find a lower bound using the Cauchy-Schwarz inequality:
\[
\begin{split}
T_2 &\geq \frac{\mu}{\pi} \int_{0}^{L/2} \theta^{2}(x)\cot^{2}(\mu x)\,dx \\
&\geq \frac{\mu}{\pi} \frac{2}{L} \Big(\int_0^{L/2} \theta(x) \cot(\mu x) dx \Big)^2 \geq \frac{2}{ L^2} I(t)^2
\end{split}
\]
Finally we have $\dfrac{dJ}{dt} \geq \dfrac{2}{L^2}{I^2}$, implying
\begin{equation}\label{gengron}
\frac{d}{dt}I(t) \geq J(0) + c_{0}\int_0^t I(t)^2dt \geq c_{0}\int_0^t I(t)^2dt,\qquad c_{0} = \frac{2}{L^{2}},
\end{equation}
From this inequality finite time blow up can be inferred in a standard way. We sketch one such argument below.

The inequality \eqref{gengron} can be equivalently written as
\begin{displaymath}
  g'' \geq 2c_{0} g (g')^{1/2},\qquad \text{where}\qquad g(t) = \int_{0}^{t} I^{2}(s)\,ds.
\end{displaymath}

Note that $\alpha:=g'(0)>0$ and $g(0)=0$. It is not difficult to show that if $h$ satisfies the second order differential equation $
h'' = 2c_{0} h (h')^{1/2}$ for the initial data
$h'(0)=\alpha \mbox{ and }h(0)=0$, then $g\geq h$ as long as these functions are well defined.

By substitution $f:=h'$, we obtain $\sqrt{f}\dfrac{df}{dh}=2c_0h$. Solving this equation we see
that for $t>0$, $h$ satisfies
$$
  (h'(t))^{3/2} = \alpha^{3/2} + \tfrac{3}{2} c_{0} h^{2}(t).
$$ From the differential inequality $h'(t)\geq \Big[\tfrac{3}{2} c_{0} h^{2}(t)\Big]^{2/3}$ together with $h(t_0)>0$ for any $t_0>0$, we get
finite time blow up for $h$. Indeed,
for any fixed $t_0>0$, we have $h(t_0)\geq \alpha\cdot t_0$ from $h'(t)\geq\alpha$. Then, for $t\geq t_0$,
a simple computation leads to
$
  h(t) \geq \Bigl( ( \alpha t_0)^{-1/3} - \tfrac{1}{3} \bigl( \tfrac{3}{2} c_{0} \bigr)^{2/3} (t-t_0) \Big)^{-3}.
$
It immediately implies the finite-time blowup of $I = (g')^{1/2}$. 
The proof of Theorem \ref{thm_1d_blowup} is complete.

\end{proof}

\section{Appendix}

In this appendix we prove bounds on the kinetic energy as well as other norms of the velocity $u$ in the periodic HL-model, showing that they remain finite at the blow up time.
We will consider a class of initial data which is smaller than the entire set for which we prove finite time blow up. With more technical effort, we can
generalize the bounds below to a larger class; however our main point is to provide blow up examples with finite kinetic energy. We therefore choose to work with the smaller
class to reduce technicalities.

In the blow up scenario for the $L$-periodic HL model, we assumed that $\omega_0$ and $\theta_{0x}$ are smooth and odd with respect to $0$ and $L/2$, and $\omega_0,\theta_{0x} \geq 0$
on $[0,\frac12 L].$ Let us assume, in addition, that $\omega_0$ and $\theta_{0x}$ vanish on $[L/4,L/2].$
Then Lemma~\ref{l3a} implies that this remains true for all times.

\begin{prop}\label{hlenergy}
Suppose that the initial data to the periodic HL-model are as described above. Then while the solution remains regular, the fluid velocity $u$ satisfies
\begin{equation}\label{ubound91}
\|u(\cdot, t)\|_{BMO} \leq C\left(\|\omega_0\|_{L^1}+\|\theta_0\|_{L^\infty}t\right).
\end{equation}
In particular, the right hand side of \eqref{ubound91} also bounds any $L^p,$ $p<\infty,$ norm $\|u\|_{L^p},$ with a constant dependent on $p.$
\end{prop}
\begin{proof}
The first observation is that
\[ \partial_t \int_{0}^{L/2} \omega(x,t)\,dx = \int_0^{L/2} \theta_x \,dx - \int_0^{L/2} u \omega_x \,dx. \]
The first integral on the right hand side is bounded by $2\|\theta_0\|_{L^\infty}.$ In the second integral, while the
solution remains smooth, we can integrate by parts to obtain
\begin{eqnarray*}
  - L\int_0^{L/2} u \omega_x \,dx = L\int_0^{L/2} u_x \omega \,dx = \\
\int_0^{L/2} \omega(x)  P.V. \int_0^L \cot (\mu(x-y))\omega(y)\,dydx = \\
  \int_0^{L/2} \omega(x)  P.V. \int_0^{L/2} \cot (\mu(x-y))\omega(y)\,dydx \\
  - \int_0^{L/2} \omega(x) \int_0^{L/2} \cot (\mu(x+y))\omega(y)\,dydx,
  \end{eqnarray*}
where $\mu = \pi/L$ as before.
Here we used $\omega(L-x)= -\omega(x)$ when transforming the last integral.
Now
\[ \int_0^{L/2} P.V. \int_0^{L/2} \cot (\mu(x-y))\omega(y)\omega(x)\,dy dx =0 \]
by symmetry.
On the other hand,
\[ \int_0^{L/2}  \int_0^{L/2} \cot (\mu(x+y))\omega(y)\omega(x)\,dydx \geq 0\]
since the support of $\omega$ lies in $[0,L/4]$ for all times by our choice of the initial data and Lemma~\ref{l3a}.
Thus we get
\[  \int_{0}^{L/2} \omega(x,t)\,dx \leq \int_{0}^{L/2} \omega_0(x)\,dx + 2\|\theta_0\|_{L^\infty} t. \]
Due to solution symmetries and choice of the initial data, this implies the key estimate
\begin{equation}\label{oml1} \|\omega(\cdot, t)\|_{L^1} \leq \|\omega_0\|_{L^1} + 2\|\theta_0\|_{L^\infty} t \end{equation}
for all $t$ while the solution remains smooth.
Now the bound on $\|u\|_{L^2}$ follows by a simple direct estimate using \eqref{eqn_m_hl_u}:
\[ u(x) = \frac{1}{\pi} \int_0^L \omega(y) \log|\sin[\mu(x-y)]|\,dy. \]
More generally, \eqref{eqn_m_hl_u} also implies the bound $\|u\|_{BMO} \leq C \|\omega\|_{L^1}.$
Indeed, it is not difficult to see that the function $\log|\sin[\mu(x-y)]|$ belongs to BMO (see e.g. \cite{Stein} for the
$\log|x|$ case). Then the formula \eqref{eqn_m_hl_u} above and \eqref{oml1} show that $u$ is a convolution of BMO function
with an $L^1$ density, leading to $\|u\|_{BMO} \leq C \|\omega\|_{L^1}$ bound and so to \eqref{ubound91}.
The bound on the $L^p$ norm of $u$ for any $p < \infty$ follows from the BMO bound.
\end{proof}

\section*{Acknowledgement}

KC has been partially supported by the NSF FRG Grant DMS-1159133. TH and GL have been partially supported by the NSF FRG Grant DMS-1159138 and the DOE Grant DE-FG02-06ER25727.
AK has been partially supported by the NSF-DMS Grants 1412023 and 1535653, and he also acknowledges support from the Guggenheim Fellowship. VS has been partially supported by the NSF-DMS Grants 1101428 and 1159376.
YY has been partially supported by the NSF-DMS Grants 1104415, 1159133 and 1411857.


\begin{thebibliography}{99}

\bibitem[Beale et~al.(1984)Beale, Kato, and Majda]{bkm1984} J. T. Beale, T. Kato, and A. Majda, ``Remarks on the breakdown of smooth solutions for the 3-D Euler equations'', Commun. Math. Phys. 94, 61--66 (1984).

\bibitem[Bourgain and Li(2013)Bourgain and Li]{bl2013}
 J.~Bourgain, D.~Li,
``Strong ill-posedness of the incompressible Euler equation in borderline Sobolev spaces'', arXiv:1307.7090.

\bibitem[Bourgain and Li(2014)Bourgain and Li]{bl2014}
 J.~Bourgain, D.~Li,
``Strong illposedness of the incompressible Euler equation in integer $C^m$ spaces'', arXiv:1405.2847.

\bibitem[Burgers(1948)Burgers]{burgers}
 J.~M.~Burgers, ``A mathematical model illustrating the theory of turbulence'', edited by Richard von Mises and Theodore von K\'{a}rm\'{a}n, Advances in Applied Mechanics, pp. 171--199. Academic Press, Inc., New York, N. Y. (1948).


\bibitem[Castro and C\'{o}rdoba(2010)Castro and C\'{o}rdoba]{cc2010} A. Castro and D. C\'{o}rdoba, ``Infinite energy solutions of the surface quasi-geostrophic equation'', Adv. Math. 225, 1820--1829 (2010).



\bibitem[Choi et~al.(2015)Choi, Kiselev, and Yao]{cky2015} K. Choi, A. Kiselev, and Y. Yao, ``Finite time blow up
for a 1D model of 2D Boussinesq system'', Commun. Math. Phys. 334, 1667--1679, (2015).

\bibitem[Constantin(2000)Constantin]{Const} P.~Constantin, ``The Euler equations and nonlocal conservative Riccati equations'',
Internat. Math. Res. Notices, no. 9, 455--465 (2000).

\bibitem[Constantin et~al.(1985)Constantin, Lax, and Majda]{clm1985} P. Constantin, P. D. Lax, and A. Majda, ``A simple one-dimensional model for the three-dimensional vorticity equation'', Comm. Pure Appl. Math. 38, 715--724 (1985).

\bibitem[Constantin et~al.(1994)Constantin, Majda and Tabak]{cmt1994}
P.~Constantin, A.~Majda, E.~Tabak, ``Formation of strong fronts in the 2-D quasigeostrophic thermal active scalar'', Nonlinearity 7, 1495--1533 (1994).

\bibitem[C\'{o}rdoba et~al.(2005)C\'{o}rdoba, C\'{o}rdoba, and Fontelos]{ccf2005} A. C\'{o}rdoba, D. C\'{o}rdoba, and M. A. Fontelos, ``Formation of singularities for a transport equation with nonlocal velocity'', Ann. Math. 162, 1377--1389 (2005).


\bibitem[Danchin(2013)Danchin]{danchin2013} R.~Danchin, ``Remarks on the lifespan of the solutions to some models of incompressible fluid mechanics'',
Proc. Amer. Math. Soc. 141, 1979--1993 (2013).

\bibitem[De Gregorio(1990)De Gregorio]{degreg1990} S. De Gregorio, ``On a one-dimensional model for the three-dimensional vorticity equation'', J. Stat. Phys. 59, 1251--1263 (1990).

\bibitem[De Gregorio(1996)De Gregorio]{degreg1996} S. De Gregorio, ``A partial differential equation arising in a 1D model for the 3D vorticity equation'', Math. Methods Appl. Sci. 19, 1233--1255 (1996).


\bibitem[Einstein(1926)Einstein]{einstein}
A. Einstein, ``The causes of the formation of meanders in the courses of rivers and of the so-called Baer's law'', Die Naturwissenschaften 14 (1926); English translation in {\sl Ideas and Opinions}, 1954; availble online at {\tt http://people.ucalgary.ca/$\tilde{ }$kmuldrew/river.html}.

\bibitem[Gustafson et~al.(2006)Gustafson, Kang and Tsai]{gkt2006} S.~Gustafson, K.~Kang, T.-P.~Tsai,
``Regularity criteria for suitable weak solutions of the Navier-Stokes equations near the boundary'',
J. Differ. Equations 226, 594--618 (2006).


\bibitem[Glassey(1977)Glassey]{glassey1977}
R.~T.~Glassey, ``On the blowing up of solutions to the Cauchy problem for nonlinear Schr\"{o}dinger equations'', J.~Math.~Phys. 18, 1794--1797 (1977).

\bibitem[Hopf(1950)Hopf]{hopf}
E.~Hopf, ``The partial differential equation $u_t+uu_{x}=\mu u_{xx}$'', Comm. Pure Appl. Math. 3, 201--230 (1950).

\bibitem[Hou and Luo(2013)Hou and Luo]{hl2013} T.~Y.~Hou and, G.~Luo ``On the finite-time blowup of a 1D model for the 3D incompressible Euler equations'', arXiv:1311.2613 [math.AP].


\bibitem[Kiselev and Sverak(2014)Kiselev and Sverak]{kv2014} A. Kiselev and V. Sverak, ``Small scale creation for solutions of the incompressible two dimensional Euler equation'',
Annals of Math. 180, 1205--1220, (2014).


\bibitem[Luo and Hou(2014a)Luo and Hou]{lh2014a} G. Luo and T. Y. Hou,
``Potentially Singular Solutions of the 3D Axisymmetric Euler Equations'',
PNAS
111,
12968--12973, (2014).


\bibitem[Luo and Hou(2014b)Luo and Hou]{lh2014b} G. Luo and T. Y. Hou,
``Toward the Finite-Time Blowup of the 3D Incompressible Euler Equations: a Numerical Investigation''.
SIAM Multiscale Modeling and Simulation
12,
1722--1776, (2014).


\bibitem[Majda and Bertozzi(2002)Majda and Bertozzi]{mb2002} A. Majda and A. Bertozzi, Vorticity and Incompressible Flow, Cambridge University Press (2002).

\bibitem[Okamoto et~al.(2008)Okamoto, Sakajo, and Wunsch]{osw2008} H. Okamoto, T. Sakajo, and M. Wunsch, ``On a generalization of the Constantin-Lax-Majda equation'', Nonlinearity 21, 2447--2461 (2008).

\bibitem[Prandtl(1952)Prandtl]{prandtl} L. Prandtl, Essentials of Fluid Dynamics, London (Blackie and Son) (1952).

\bibitem[Seregin(2002)Seregin]{seregin}
G.~A.~Seregin,
``Local regularity of suitable weak solutions to the Navier-Stokes equations near the boundary'',
J. Math. Fluid Mech. 4, 1--29 (2002).

\bibitem[Stein(1993)Stein]{Stein} E. M. Stein, Harmonic Analysis: Real-Variable Methods, Orthogonality, and Oscillatory Integrals, Princeton (Princetion University Press) (1993). 

\bibitem[Vlasov et~al.(1971)Vlasov, Petrishchev, and Talanov]{talanov1971}
    S.~N.~Vlasov, V.~A.~Petrishchev, and V.~I.~Talanov, ``Averaged description of wave beams in linear and nonlinear media (the method of moments)'', Radiophys. Quantum. El. (translated from Izvestiya Vysshikh Uchebnykh Zavedenii, Radiofizika) 14, 1062--1070 (1971).

\bibitem[Yudovich(1963)Yudovich]{yudovich1963}
 V.~I.~Yudovich, ``Non-stationary flow of an ideal incompressible liquid'', USSR Comput. Math. Math. Phys. 3, 1407--1456 (1963).



\end{thebibliography}
\end{document}